\documentclass[journal,twocolumn,10pt,letter]{IEEEtran}

\IEEEoverridecommandlockouts
\ifCLASSINFOpdf
\else
\fi

\usepackage{bbold}
\usepackage{amsfonts}
\usepackage{nonfloat}
\usepackage{amssymb}
\usepackage[cmex10]{amsmath}
\usepackage{amsthm}
\usepackage{amssymb}
\usepackage{mathrsfs}
\usepackage{amsbsy}
\usepackage{setspace}
\usepackage{graphicx}
\usepackage{newclude}
\usepackage{latexsym}
\usepackage{lettrine} % I added this packagehttps://www.overleaf.com/project/5caf4bdce350a8305ec7a022
\usepackage{flushend}

\usepackage{authblk}

\usepackage{scalerel,stackengine}
\stackMath

\usepackage{amsthm}

\newtheoremstyle{named}{}{}{\itshape}{}{\bfseries}{.}{.5em}{\thmnote{#3's }#1}
\theoremstyle{named}

\usepackage{amsthm}

\theoremstyle{definition}

\newtheorem{definition}{Definition}

\usepackage[table]{xcolor}
\usepackage{multirow}
\usepackage{rotating}
\usepackage{booktabs}
\usepackage{bbm} % indicator function
\usepackage{amsmath,epsfig,cite,amsfonts,psfrag}
\usepackage{lipsum}
\usepackage{cuted}
\usepackage{color}

\usepackage{epstopdf}
\usepackage{float}
\usepackage{mathtools}
\usepackage{bbm}
\usepackage{atbegshi}% http://ctan.org/pkg/atbegshi
\usepackage{balance}
\usepackage[utf8]{inputenc}
\usepackage{multirow}

\renewcommand{\tablename}{Algorithm}

\usepackage{accents}
\newlength{\dhatheight}

\allowdisplaybreaks

\newtheorem{lemma}{Lemma}

\graphicspath{{figures/}}
\allowdisplaybreaks

\makeatletter
\floatstyle{plain}
\newfloat{twocolequfloat}{b}{zzz}
\floatname{twocolequfloat}{Equation}
\newtheorem{Theorem}{Theorem}

\makeatother

\begin{document}

%\title{\LARGE  Art Gallery Problem: A Novel Solution, Analysis, and Application on Indoor Placement of Visible Light Communication Nodes  }

%\title{ Indoor VLC network planning via LoS graph modeling}
%\title{ VLC network planning:the basics of indoor placement via LoS graph modeling}

%\title{ VLC network planning: indoor placement via LoS graph modeling}
%\title{ VLC network planning: indoor placement optimization  via LoS graph modeling}
%\title{ VLC network planning: fundamentals of indoor placement optimization }
%\title{ Visible Light Communication Network: Indoor Planning Optimization }
%\title{ Precise Indoor Positioning:  Deployment of Reference Nodes to Guarantee LoS Condition%DetaiPRN planning of reference nodes placement to ensure LoS condition for precise indoor positioning %Indoor Planning for Visible Light Positioning Networks \vspace{-2mm}
%}
%\title{Ensuring Line-of-Sight Wireless Coverage in Stochastic Environments}

%\title{Wireless Access Point Placement to Ensure Line-of-Sight  Coverage in Stochastic Environments}

\title{Guaranteeing Line-of-Sight Wireless Connectivity in Stochastic Environments with Random Obstacles}

    \author{Mohsen Abedi, Alexis A. Dowhuszko, Mehdi Sookhak, and Risto Wichman\thanks{M. Abedi, A. Dowhuszko, and R. Wichman are with the Dept. Information and Communication Engineering, Aalto University, 02150 Espoo, Finland. Email: \{mohsen.abedi;\,alexis.dowhuszko;\,risto.whichman\}@aalto.fi \par M. Sookhak is with the Dept. Computer Science, Texas A\&M University-Corpus Christi, Texas, USA. Email:
mehdi.sookhak@tamucc.edu}}

%\author{Mohsen Abedi, Alexis A. Dowhuszko,~\IEEEmembership{Senior Member,~IEEE}, Mehdi Sookhak, and Risto Wichman

%\thanks{M. Abedi, A. A. Dowhuszko, and R. Wichman are with the Department of Information and Communications Engineering, Aalto University, Espoo, 02150, Finland. Email: \{mohsen.abedi;\,alexis.dowhuszko;\,risto.wichman\}@aalto.fi \par M. Sookhak is with the Dept. Computer Science, Texas A\&M University-Corpus Christi, Texas, USA.}
%}

\maketitle

\begin{abstract}

Advancements in high-frequency communication technologies using millimeter waves~(mmWave), Tera-Hertz~(THz), and optical wireless frequency bands are key for extending wireless connectivity beyond 5G. These technologies offer a broader spectrum than the one available on low- and mid-bands, enabling ultra-high-speed data rates, higher device density, enhanced security, and improved positioning accuracy. However, their performance relies heavily on clear Line-of-Sight~(LoS) conditions, as Non-LoS components are significantly weaker, making blockages a major challenge to ensure suitable received signal power.
This paper addresses this limitation by identifying the minimum number and optimal placement of access points (APs) needed to ensure LoS connectivity in stochastic/dynamic environments with random obstacle locations. To achieve this, the stochastic environment is carefully modeled as a graph, where the nodes represent sub-polygons of layout realizations, and the edges capture the visibility overlaps between them. By employing maximal clique clustering and maximum clique packing methods over this graph, the proposed approach determines the AP placement locations that guarantee either full LoS coverage or controlled LoS gaps, while seamlessly adapting to the stochastic variability in obstacle locations. Simulations results in a representative stochastic environment demonstrate a $25\%$ reduction in the required number of APs, achieving a tolerable $5\%$ coverage gap compared to AP deployment optimized for full LoS coverage. %, highlighting the efficiency of the proposed detailed planning algorithm.

\thispagestyle{empty}
\pagestyle{empty}
 %Specially, the new method superiority in service areas with harsh irregularity in demand distributions is highlighted.
\end{abstract}

\vspace{0mm}

%\IEEEpeerreviewmaketitle
%%%%%%%%%%%%%%%%%%%%%%%%%%%%%%%%%%%%%%%%%%%%%%%%%%%%%%%%%%%%%%%%%%%%%%%%%%%%%%%%%%%%%%%%%%%%%%%%%%%%%%%%%%%%%%%%%%%%%%%%%%%%%%%%%%%%

\vspace{-2mm}
\section{Introduction} %\textcolor{red}{Alexis will work here}}
%$\bullet$ 5G and 6G peak performance relies on LoS condition

To advance wireless connectivity beyond 5G, high-frequency radio technologies for wireless communications in the millimeter waves~(mmWave) and Tera-Hertz~(THz) frequency bands, as well as optical wireless technologies for indoor access over infrared and visible light bands, are gaining significant attention~\cite{rappaport2019wireless}. These technologies offer key advantages over low- and mid-band Radio Frequency~(RF) communications, including wider bandwidths for ultra-fast data transmission, higher device density support, improved security, and reduced susceptibility to interference~\cite{zhang2022mmwave, moon20226g, obeed2019optimizing}. 
These high-frequency communication technologies enable a diverse range of applications, including  mmWave/THz wireless access~\cite{al2020probability}, industrial automation and connected robotics~\cite{ ju2024statistical}, autonomous vehicles~\cite{kong2024survey, mollah2023mmwave}, Unmanned Aerial Vehicle~(UAV) networks~\cite{cheng2022channel}, Extended Reality (XR)~\cite{marinvsek2024mmwave, ma2024qoe}, and land mobile satellite systems~\cite{redondi2024survey}. However, for suitable performance across these applications, maintaining a clear Line-of-Sight~(LoS) link between transmitter and receiver is critical.

Line-of-Sight communication has been positioned to play a crucial role in shaping next-generation wireless technologies, as Non-Line-of-Sight~(NLoS) signals in high frequency bands often suffer from excessive weakening or severe penetration losses when obstructed by various materials~\cite{zhao201328}. To compensate the high path loss that mmWave and THz signals experience during propagation, high-gain directional antennas are needed. Although, these type of antennas can potentially enhance throughput, they make mmWave and THz communications even more vulnerable to blockages~\cite{liu2019analysis}. In outdoor environments, these challenges are compounded by blockages from buildings, leading to pronounced differences between LoS and NLoS path loss models. Similar to RF bands in lower frequency bands, LoS paths in mmWave communication typically exhibit an approximate path loss exponent of 2, while NLoS paths often have exponents close to 4 or higher. 

This disparity between LoS and NLoS propagation results in much better Signal-to-Interference-plus-Noise Ratio~(SINR) in the former case~\cite{bai2014coverage}. For example, the study in~\cite{korrai2019performance} reports a $30\%$ improvement in data rates at a carrier frequency of 73\,GHz for long-range LoS links compared to NLoS ones. Similarly, the study in~\cite{sur2017wifi} demonstrates around $200\%$ increase in data rates in a short-range LoS 60 GHz WiFi setup compared to a NLoS one. On the other hand,
optical wireless signals, such as those used in Free Space Optics~(FSO) and Visible Light Communication~(VLC), present additional challenges. Optical wireless signals are entirely obstructed by opaque obstacles, resulting in a significant degradation of the achievable data rates, as NLoS links experience notably weaker signal strengths in reception compared to LoS~\cite{dowh2020a}. 

Beyond communication applications, maintaining a LoS propagation link is particularly critical for positioning systems. Various 5G positioning techniques such as Round Trip Time~(RTT), Angle of Arrival~(AoA), Angle of Departure~(AoD), Time of Arrival~(ToA), and Time Difference of Arrival~(TDoA), rely heavily on LoS propagation to ensure unbiased and accurate position estimation~\cite{keat2019}. This reliance stems from the fact that NLoS signals traverse longer and more complex paths, introducing systematic errors that degrade the precision when estimating the User Equipment~(UE) position. Additionally, ToA and TDoA measurements based on 5G mmWave technology offers finer time resolution when estimating the actual propagation time of the wireless signals, enabling highly accurate positioning thanks to the significantly wider bandwidths~(up to 400\,MHz) compared to the typically narrower bandwidths of sub-6 GHz bands~(100\,MHz or less)~\cite{holm2020}.  
 The adoption of optical wireless technology in the next generation of mobile communication systems will further improve positioning accuracy, enabling precision at the centimeter or even sub-centimeter level. However, leveraging these advanced technologies for accurate positioning  face significant challenges in the absence of a LoS condition~\cite{ren2021}. This fact has motivated significant research efforts to detect the existence of LoS/NLoS when detecting received signals~\cite{choi18}.

%$\bullet$ probabilistic model for LoS

The advantages of LoS  over NLoS propagation have motivated researchers in academia and industry to develop numerous empirical models for capturing LoS probability, with most of them tailored to specific applications~\cite{al2020probability}. The International Telecommunication Union~(ITU) has proposed the most widely used statistical model for LoS probability, specifically designed for macro and micro cellular networks including defined constraints on permissible node heights~\cite{series2017guidelines}. With the advent of mmWave technology in 5G, researchers and industry professionals have conducted extensive  studies to analyze and determine the LoS probability that high-frequency radio signals experience for various distance between transmitter and receiver~\cite{haneda20165g, jassim2024new}, using  machine learning techniques~\cite{ wu2023machine}, laser scanning~\cite{jarvelainen2016evaluation}, empirical data~\cite{ barbiroli2024simple}, and as function of the spatial density, size, and height of the buildings in urban environments~\cite{ gapeyenko2021line,fujimura2023line,saboor2023geometry}. These probabilistic models are utilized for network dimensioning, enhancing wireless coverage, improving system performance, and informing the design of communication protocols. However, the complexities and dynamics of urban, rural, and indoor environments present significant challenges to carry out this task, making the development of effective strategies to maintain a consistent LoS propagation link an area that requires further investigation. % in the literature.

Efficient network deployment is a crucial approach to maximize the LoS coverage probability in wireless networks. Focusing on mmWave networks, the authors in \cite{liu2021maximizing} address the multi-AP placement problem to optimize LoS coverage in rectangular or rectilinear-shaped rooms without obstacles by formulating it as a thinnest covering problem and employing a shadowing elimination search algorithm.
In \cite{chatterjee2021joint}, the authors model the AP placement problem as a stochastic optimization problem, which is approximated using the Sample Average Approximation~(SAA) method. After linearizing the problem, the authors solved it optimally using CPLEX, a commercial optimization solver. Similarly, \cite{topal2023optimal} formulates the AP deployment problem for indoor environments with static users as a mixed-binary programming problem and applies the Branch-and-Bound Algorithm  (BBA) to derive optimal deployment.
The authors in \cite{tsai2024novel} proposed a graph-based approach to determine optimal AP placement on the vertices of the layout  by constructing a graph based on the open visibility between every pair of vertices. In \cite{abedi2021visible},  a novel graph construction method was proposed to model LoS optical wireless coverage in a static indoor environment without obstacles. By clustering this graph, the minimum number of APs and their precise locations within the layout are identified to provide full LoS coverage. Later, this approach was  extended to maintain LoS optical wireless backhauling between APs in a similar environment\cite{abedi2024indoor}. However, developing an optimal placement strategy that guarantees LoS coverage for any AP range, while also accounting for uncertainties in obstacle locations, remains a notable challenge in the existing literature. This difficulty arises because the problem, even in its simplest forms, is NP-hard and requires an exhaustive search solution method. This type of solution to determine the optimal placement in real-world scenarios is extremely complicated and often demands an impractical amount of computational resources.

In this paper, we propose a novel method to determine the minimum number of APs and their optimal locations to ensure LoS connectivity for users that take random positions in an stochastic environment with random obtacles. The approach involves developing a new graph construction and analysis framework, where the graph nodes represent sub-polygons partitioning the environment under each realization, and the edges denote the common visibility areas between sub-polygons. By partitioning this graph into the minimum number of cliques, we identify the visibility area of each clique as an AP placement location. This method guarantees LoS connectivity between a user at any random location and at least one AP, despite the random nature of obstacle locations. Unlike stochastic geometry-based approaches commonly found in the literature, this solution is deterministic, adaptable to both indoor and outdoor environments, and capable of ensuring optimal deployment across a wide range of wireless scenarios, while adhering to diverse wireless hardware and AP deployment constraints.

The contributions of this paper are summarized as follows:

\begin{itemize}
\item   The stochastic environment is represented by a set of layout realizations, each of which is partitioned into sub-polygons using a novel method called \emph{hyper-triangulation}~\cite{abedi2024indoor}. To maintain the computational efficiency of our algorithms, we determine and establish an upper bound on the parameter of the hyper-triangulation in terms of the AP range and the inradius of the convex hull of the  smallest obstacle within the environment. This parameter influences the number of sub-polygons and, subsequently, the number of nodes in the graph model.
\item Once all the layout realizations are partitioned, we create a \emph{visibility graph} whose nodes represent the resulting sub-polygons, and the edges between each node pair represent the overlap of the visibility area of the nodes. This graph will serve as the foundation for the algorithms, analysis, and optimality assessment presented in this paper.
\item We propose a \emph{maximal clique clustering} method to partition the nodes of this graph into the minimum number of cliques. We prove that the minimum number of cliques partitioning this graph represents the minimum number of APs required to ensure a full LoS coverage regardless of the location of the obstacles. Moreover, it is shown that the visibility area of each clique equivalently represents the potential placement location for a single AP. 
\item By defining an acceptable LoS coverage gap, a \emph{maximaum clique packing} method was proposed to identify the largest clique in the visibility graph and to remove a subset of nodes based on the defined coverage gap. It is also demonstrated that placing a single AP within the visibility area of each clique identified by the algorithm ensures the acceptable LoS coverage gap for the stochastic environment with the minimum number of\,APs. 
\end{itemize}

Furthermore, we demonstrate that the independence number of the visibility graph serves as a lower bound for the minimum number of APs needed to achieve a full LoS coverage in the stochastic environment. This number provides a critical benchmark for evaluating the optimality of our AP deployment strategy. In the simulation results section, we apply the proposed method to an environment with four obstacles, one of which is stochastically positioned. The results show that by allowing a tolerable coverage gap, e.g., as small as $5\%$, the number of required APs is reduced by approximately $25\%$.

The remainder of this paper is structured as follows: Section~\ref{sec:2} outlines the problem statement, providing the foundation for the study. Section~\ref{sec:3} introduces the methodology for constructing a graph model that captures the stochastic nature of the environment. Based on this model, Section~\ref{sec:4} develops the deployment strategies to ensure LoS coverage as well as an acceptable coverage gap. Section~\ref{sec:5} presents the detailed simulation results, while Section VI concludes the paper with a summary of key findings and insights.

%$\bullet$ The Use cases that require LoS links like XR, positioning , wireless sensor networks...
%\label{sec:1}
%\vspace{-0.5mm}

%$\bullet$ previous efforts on LoS coverage and other models like probabilistic LoS 

%\vspace{-2mm}
\section{Problem Statement}
\label{sec:2}
%\vspace{-0.5mm}

The maximum range of an AP operating in the mmWave or THz bands for communication or positioning purposes primarily depends on several factors, including the power, carrier frequency, hardware efficiency, atmospheric absorption, receiver sensitivity, noise and interference levels, and key antenna properties such as beamforming techniques, gain, and array size. On the other hand, in optical wireless systems such as VLC networks, the maximum range is mainly influenced by the height of the Light Emitting Diode (LED) and the Field of View (FoV) of the optical receiver \cite{abedi2024indoor}.  Fig. \ref{fig:range_definition}a) depicts the schematic of a mmWave/THz AP in an urban environment, communicating with a user in a LoS condition within its range. In contrast, two other users experience signal loss—one due to being outside the coverage range and the other due to obstruction caused by a vehicle. Similarly, Fig. \ref{fig:range_definition}b) illustrates a VLC communication scenario in an indoor environment, where a VLC AP (LED lamp) communicates with a user in a LoS condition within its range. However, two other users are unable to receive any signal due to human body blockage and being located outside the range. As a result, it is crucial to consider both the limited range and the impact of stochastic or dynamic obstacles when optimizing AP placement to ensure LoS coverage.

\begin{figure}
\begin{minipage}{\columnwidth}\hspace{0.3cm}
\includegraphics[width=3.0in]{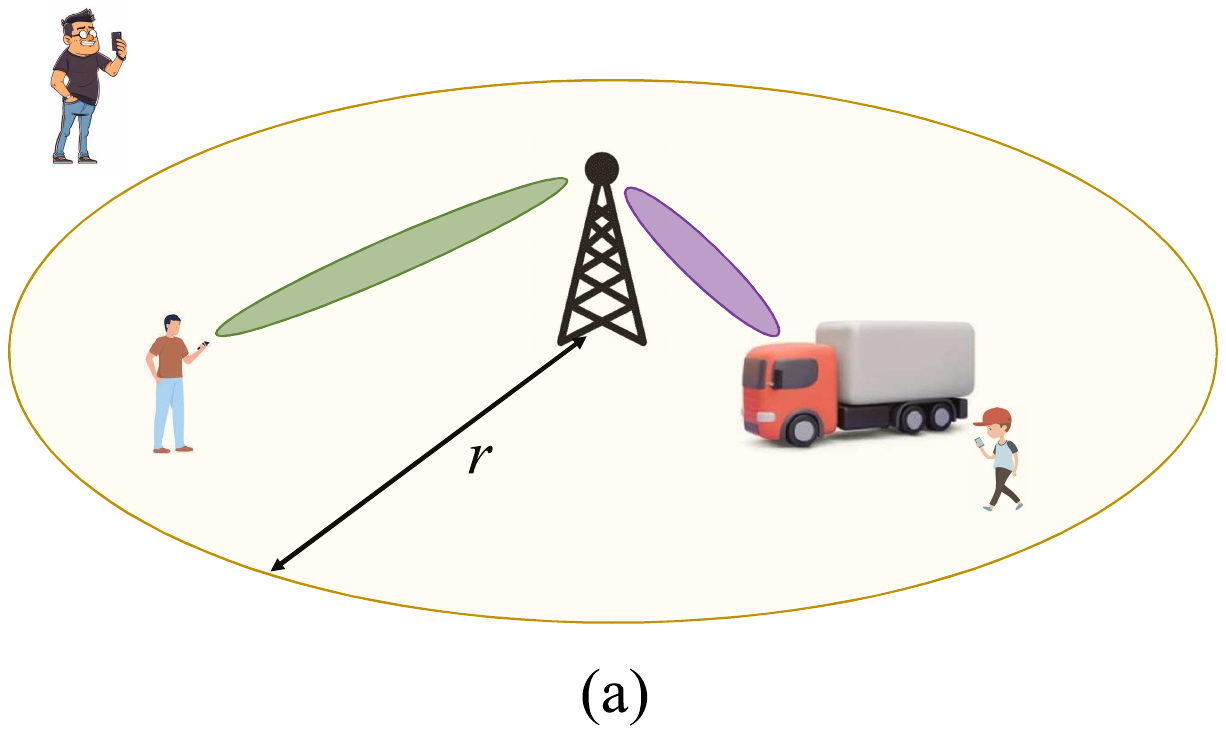}
\vspace{-5.3cm}
\end{minipage}

\begin{minipage}{\columnwidth}\hspace{0.4cm}\includegraphics[width=3.0in]{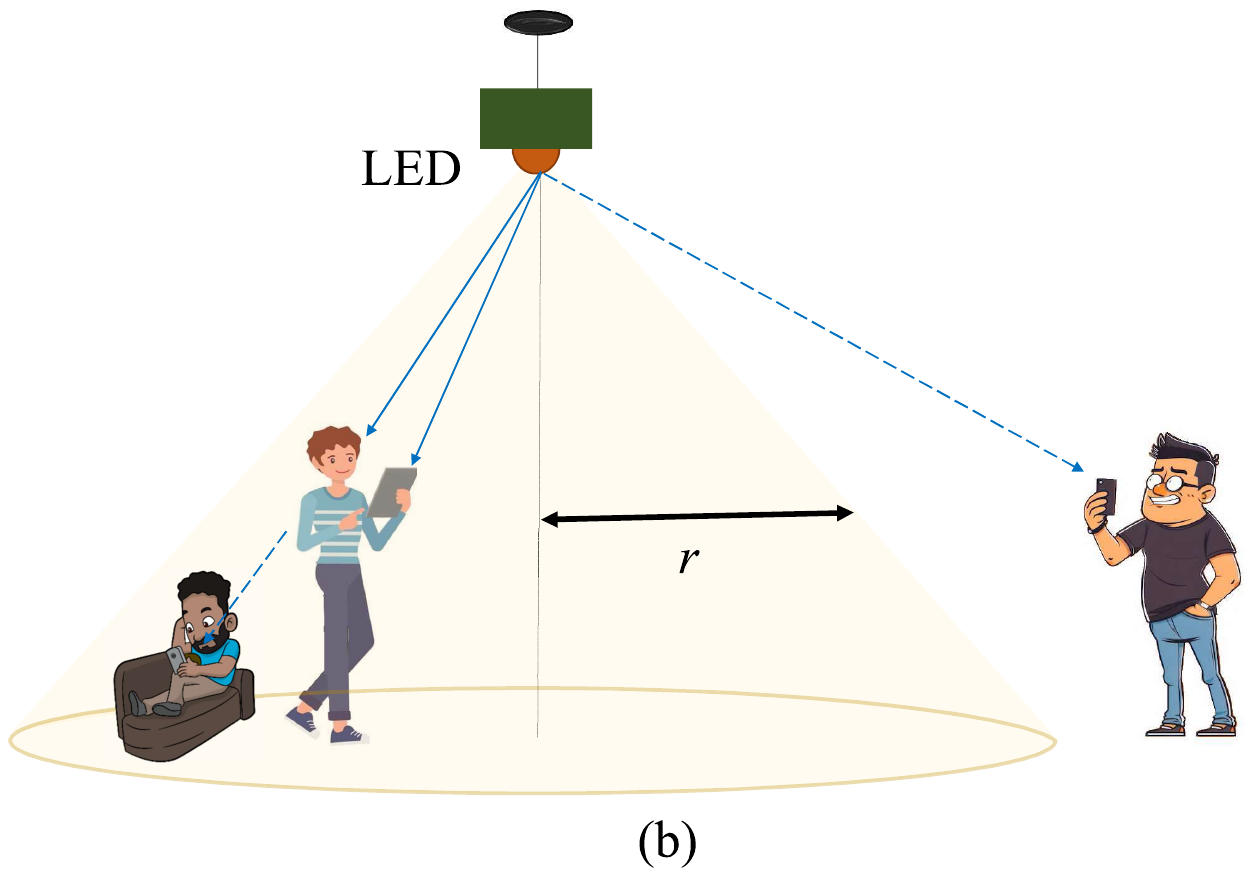}
\vspace{-4.5cm}
\end{minipage}
\caption{Sample illustration of the maximum range $r$ of: a) radio system with a mmWave/THz base station; b) optical wireless communication system using a VLC access point. In the former (latter) case, communication beyond $r$ is not possible due to signal strength would fall below the receiver sensitivity (the receiver would be placed outside the emission cone of the LED).}
\label{fig:range_definition}
\end{figure}
The problem of AP placement to ensure LoS coverage in a deterministic environment without obstacles (holes) and unlimited AP range is analogous to the \emph{art gallery problem} in computational geometry \cite{o1998computational}. The art gallery problem involves placing the minimum number of stationary, omni-directional camera guards required to cover all points within an art gallery of arbitrary layout.
By representing the walls of the art gallery as a layout with $n$ vertices, a single guard placed anywhere within a convex layout is sufficient for full coverage. This is because every point of the convex polygon is fully visible from any other point within it. Minimizing the number of camera guards required to cover a layout is generally an \mbox{NP-hard} problem. The \emph{Chvátal Theorem} provides an upper bound on the minimum number of guards needed. This theorem states that to fully cover any layout with $n$ vertices, $\lfloor n/3 \rfloor $ guards are sufficient and, in some cases, necessary. The theorem is proven by Fisk's proof, which demonstrates that the vertices of any layout are three-colorable \cite{fisk1978short}.
Alternatively, \emph{convex partitioning} methods determine guard placement by adding the smallest set of non-intersecting diagonals needed to partition the layout into the minimum number of convex sub-polygons \cite{chazelle1994decomposition}.
In this approach, placing a camera guard at the center of each convex sub-polygon ensures full visibility. However, these solutions to the art gallery problem are heuristic in nature, often deviating significantly from the optimal solution in complex layouts. Moreover, they are only applicable in deterministic environments without obstacles and assuming that the AP is not range constrained.

Unlike the assumptions made in the art gallery problem, real-world environments often include obstacles with stochastic spatial distributions or dynamic locations, which can block the LoS links between APs and users. Consider, for instance, a layout with a square-shaped obstacle positioned at two distinct locations, representing two realizations of the layout, as illustrated in Fig. \ref{fig:problem_statement}. Assuming that the APs are range unconstrained, Fig.\ref{fig:problem_statement}a) illustrates the first sample AP placement, shown by yellow triangles, that ensures LoS coverage for the first realization. Also, Fig. \ref{fig:problem_statement}b) depicts the second sample AP placement that guarantees LoS coverage for the first realization of the layout. However, in the second realization, Figures \ref{fig:problem_statement}c) and \ref{fig:problem_statement}d) demonstrate that both the first and second sample AP deployments fail to ensure seamless LoS coverage. As a result, understanding the spatial distribution or dynamics of obstacles is crucial for determining the optimal AP placement.

\begin{figure}[!t]
\centering
%\hspace{-1cm}
\advance\leftskip-1.3cm
\advance\rightskip-1cm
\includegraphics[width=8cm]{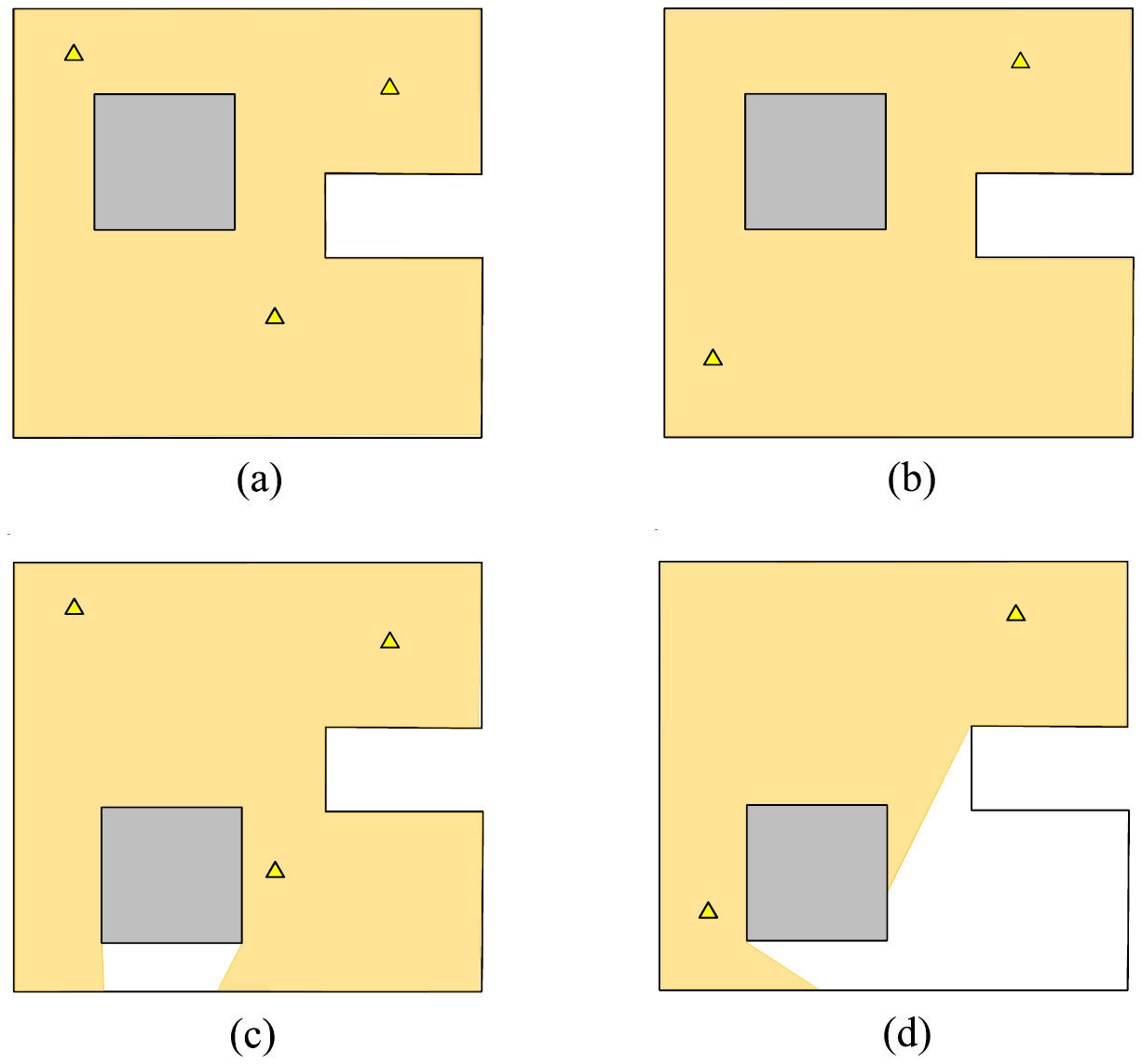} %{vtc_fig0.eps}
    \vspace{-3.3cm}\caption{ Access point placement (triangles) in a stochastic environment containing the example layout with $n=8$ vertices and a square-shaped obstacle in two location realizations. Ensuring LoS coverage in the first realization using (a) the first sample AP placement and (b) the second sample AP placement; Failure to achieve LoS coverage in the second realization using (c) the first sample AP placement and (d) the second sample AP placement. }\label{fig:problem_statement}
\vspace{-0.2cm}
\end{figure}

\begin{figure}[!t]
\centering
%\hspace{-1cm}
\advance\leftskip-1.3cm
\advance\rightskip-1cm
\includegraphics[width=8cm]{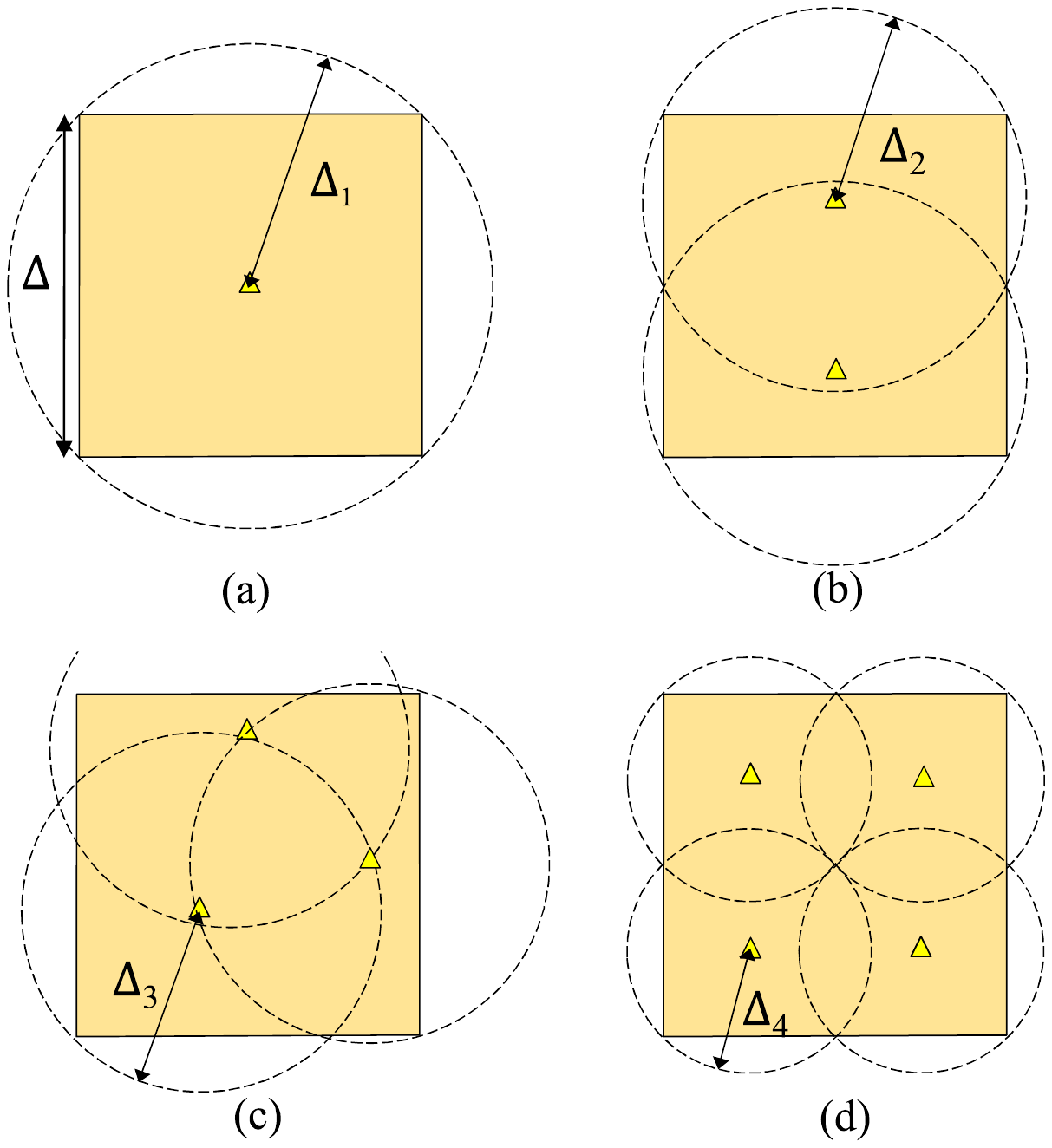} %{vtc_fig0.eps}
    \vspace{-2.2cm}\caption{ LoS coverage of the square-shaped layout with side length $\Delta$ by optimal deployment of: a) A single AP with range $r=\Delta_1=\sqrt{2}/2 \Delta$; b) two APs with range  $r=\Delta_2=\sqrt{5}/4 \Delta$; c) three APs with  range $r=\Delta_3= (\sqrt{6}/2-\sqrt{2}/2)\Delta$; and d) four APs with range $r=\Delta_4=\sqrt{2}/4\Delta$. }\label{fig:square-limited-range}
\vspace{-0.2cm}
\end{figure}

The limited range of wireless APs further complicates the task of optimal placement. In a square-shaped room with a side length of $\Delta$ and no obstacles, as shown in Fig.~\ref{fig:square-limited-range}(a), placing a single AP at the center with a range of $r=\sqrt{2}/2 \Delta$ or greater is sufficient to achieve LoS coverage. Additionally, as illustrated in Fig.~\ref{fig:square-limited-range}(b), deploying two APs with a range of  $\sqrt{5}/4 \Delta\leq r<\sqrt{2}/2 \Delta$ is sufficient to cover the entire area. Interestingly, Fig.~\ref{fig:square-limited-range}(c) shows that when the range satisfies $(\sqrt{6}/2-\sqrt{2}/2)\Delta \leq r < \sqrt{5}/4 \Delta$, three  APs are both necessary and sufficient to achieve LoS coverage.
Lastly, Fig.~\ref{fig:square-limited-range}(d) shows that four APs are needed when the AP range is limited to $\sqrt{2}/4\Delta \leq r< (\sqrt{6}/2-\sqrt{2}/2)\Delta$. 
As a result, since deploying APs to ensure LoS coverage in a square room is already challenging, it is reasonable to expect that irregular layouts will further complicate this task. 

Developing a deterministic solution that identifies the minimum number of APs and their precise locations to ensure LoS connectivity between APs and users throughout the environment remains an open challenge. Such a solution should be capable of accommodating any layout shape, spatial distribution of obstacles, and range configurations. Furthermore,  this solution approach should be adaptable to scenarios where a certain degree of LoS coverage gap is acceptable within the environment. The following sections provide a detailed, step-by-step description of this solution approach.

\vspace{0mm}
\section{Graph Modeling of Stochastic Environments }
\label{sec:3}

To lay the groundwork for the solution approaches in this paper, we begin by defining the concept of the visibility area. This definition is then used as a basis for constructing and analyzing our graph model.

 Given a point $P$, a sub-polygon $p$, and the maximum range $r>0$, we present the underlying definitions.

\begin{definition}\label{point_visibility.def}
  $\bf{Visibility~area~of~a~ 
 ~point}$ $\mathcal{V}_r(P)$ denotes the set of all the points $Y$ such that the line segment $PY$ lies entirely inside the layout without disruption, and $\|PX\|\leq r$.
 \end{definition}So, there is no layout  edge or obstacle between $P$ and any point in $\mathcal{V}_r(P)$ while $P$  lies within the distance $r$ from any point in $\mathcal{V}_r(P)$.

\begin{definition}\label{polygon_visibility.def}
 $\bf{Visibility~area~of~a~
 ~polygon}$ $\mathcal{V}_r(p)$ refers to the set of all the points $X$, such that $X\in\mathcal{V}_r(P_i)$ for $i=1, 2,..., l$, where $P_1, P_2, ..., P_l$ denote the vertices of $p$.
\end{definition}
From Definitions \ref{point_visibility.def} and \ref{polygon_visibility.def}, it follows immediately that  $\mathcal{V}_r(p)=\bigcap_{i=1}^l  \mathcal{V}_r(P_i)$. The definition of visibility area is visualized in Fig.~\ref{fig:visibility-definition} for the example environment and the AP range $r_0$. Given a point $P_1$ within the environment,  the yellow areas of Fig.~\ref{fig:visibility-definition}(a) illustrate $\mathcal{V}_{r=r_0}(P_1)$.  
While, the yellow area of Fig.~\ref{fig:visibility-definition}(b) represents $\mathcal{V}_{r=r_0}(p)$, wherein $p$ is a triangle with vertices $P_1$, $P_2$ and $P_3$. It can be observed here that
all vertices of $p$ are openly visible and lie within $r_0$ from any point in $\mathcal{V}_{r=r_0}(p)$. Similarly, we define the visibility area of a set of sub-polygons $c=[p_1,p_2,...,p_k]$ as $\mathcal{V}_{r}(c)=\bigcap_{i=1}^k  \mathcal{V}_{r}(p_i)$. This definition will later be utilized to determine the locations of the APs. 

\begin{figure}[!t]
\centering
%\hspace{-1cm}
\advance\leftskip-1.3cm
\advance\rightskip-1cm
\includegraphics[width=8cm]{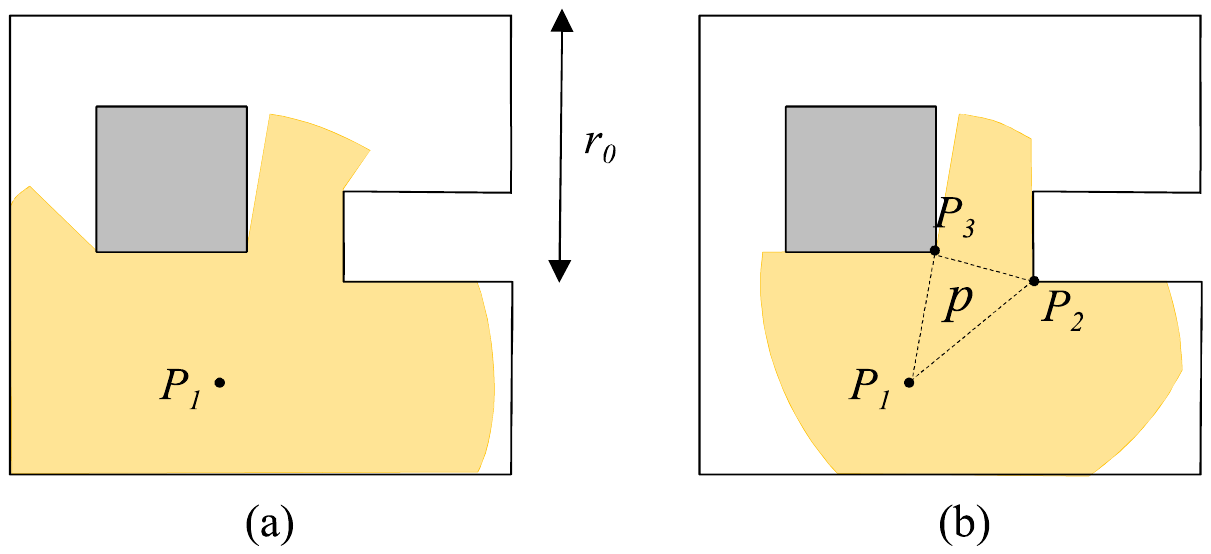} %{vtc_fig0.eps}
    \vspace{-7cm}\caption{ Definition of visibility area in the example layout with a square-shaped obstacle and AP range $r$; a)  visibility area of the point $P_1$, and b) visibility area of the triangle $p=\triangle P_1P_2P_3$. }\label{fig:visibility-definition}
\vspace{-0.1cm}
\end{figure}

\begin{figure}[!t]
\centering
%\hspace{-1cm}
\advance\leftskip-1cm
\advance\rightskip-1cm
\includegraphics[width=8cm]{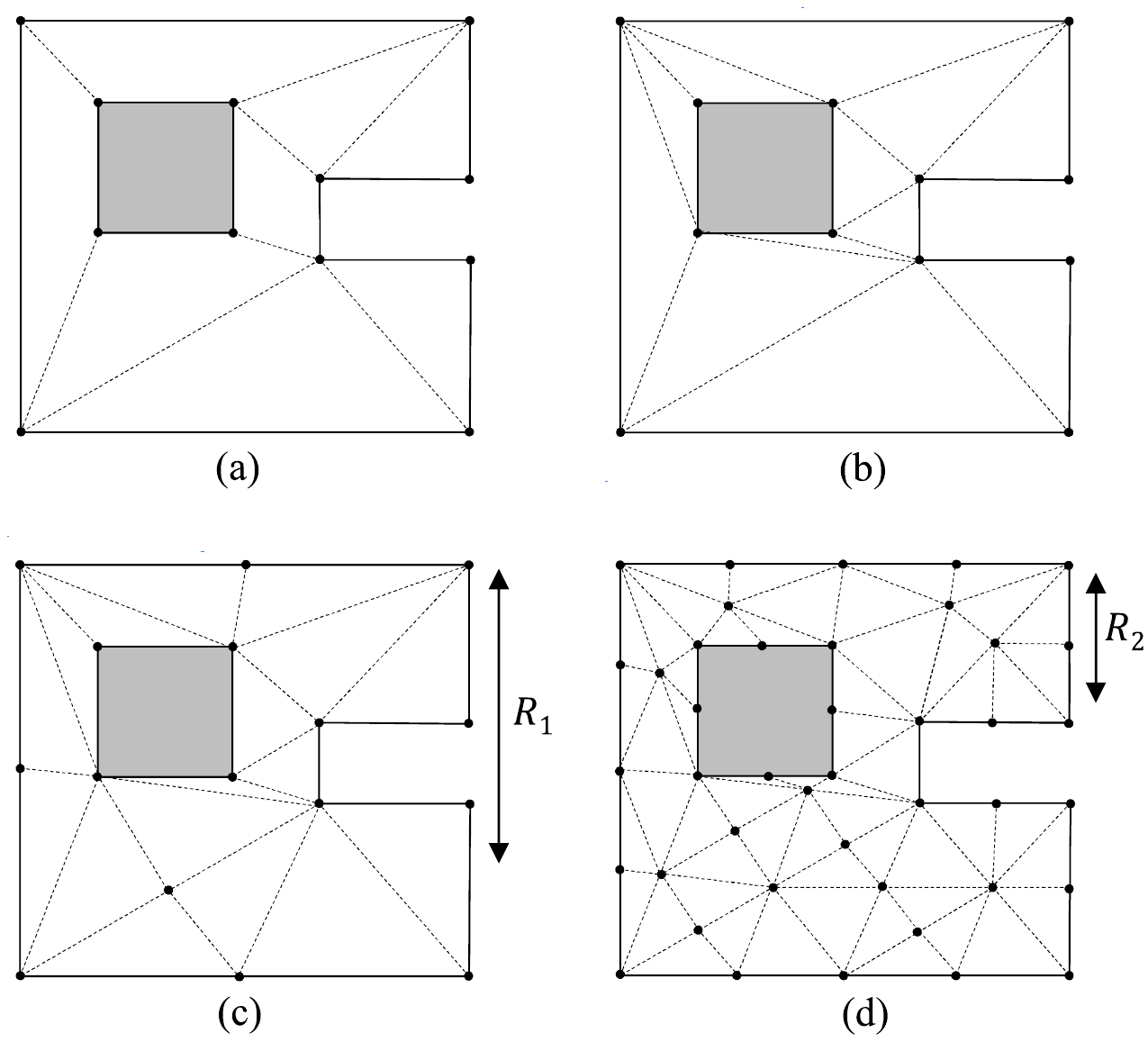} %{vtc_fig0.eps}
    \vspace{-3.4cm}\caption{Layout partitioning methods; a) Convex partitioning, b)   triangulation or hyper triangulation $\mathcal{HT}(R=\infty, \mathcal{L})$, c) hyper triangulation $\mathcal{HT}(R=R_1, \mathcal{L})$, and d)  and hyper triangulation $\mathcal{HT}(R=R_2, \mathcal{L})$.}\label{fig:hypertriangulation}
\vspace{-0.2cm}
\end{figure}

\vspace{0mm}
\begin{figure*}[!t]
\centering
%\hspace{-1cm}
\advance\leftskip-0.8cm
\advance\rightskip-1cm
\includegraphics[width=18.5cm]{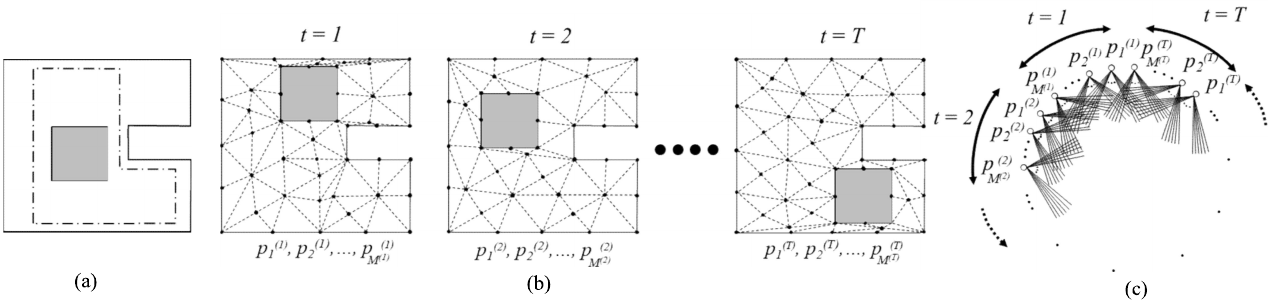} %{vtc_fig0.eps}
    \vspace{-19.7cm}\caption{ The process of creating the visibility graph; a) The example layout with a square-shaped obstacle of random location,
    b)  a set of $T$ realizations and the corresponding hyper-triangulation, and c) the resulting PV graph. }\label{fig:graph-large}
\vspace{-0.2cm}
\end{figure*}

\subsection{Pre-processing  the stochastic environment}

    Creating an equivalent graph model starts with partitioning the environment. The simplest partitioning method is convex partitioning, where the layout is divided into a set of convex sub-polygons by adding diagonals. Fig. \ref{fig:hypertriangulation}a) illustrates an example of convex partitioning, where the layout is divided into eight sub-polygons. Triangles are particularly appealing among all types of sub-polygons due to their inherent convexity and their ability to effectively partition any layout. Notably, any layout can be divided into a set of triangles through a process known as \emph{triangulation}~\cite{de1997computational}. Fig. \ref{fig:hypertriangulation}(b) represents an example of triangulation with $n=12$ triangles. However, it may be necessary to use smaller triangles depending on the network requirements or the size of the obstacles. 
 \begin{definition}\label{hyper-triangulation.def} $\bf{Hyper~Triangulation~method}$  $\mathcal{HT}(R, {\mathcal{L}})$ begins with triangulation of the environment $\mathcal{L}$ and proceeds by connecting the midpoint of the largest side of each triangle to its opposite vertex, until no triangle remains with a side larger than a target parameter $R$.
\end{definition}

Fig. \ref{fig:hypertriangulation}(c) and (d) 
 represent hyper-triangulation of the layout with target parameter set as $R=R_1$ and $R_2$, respectively, where $R_1>R_2$. It is observed here that all the triangle sides are smaller than $R_1$ in Fig. \ref{fig:hypertriangulation}(c) and $R_2$ in Fig. \ref{fig:hypertriangulation}(d). Therefore, as $R$ decreases in $\mathcal{HT}(R,\mathcal{L})$, the total number of triangles $M$ (the total number of nodes in the graph) increases with an approximate rate of $1/R^2$ for small values of $R$. Therefore, although a larger value of $R$ is preferred to reduce the computational complexity, it is necessary to establish several upper bounds to $R$. These bounds, derived through Lemmas \ref{lemma:non-empty-visibility}, \ref{lemma:Covering-visibility}, and \ref{Lemma:visibility_to_inside_triangle_with_obstacle}, serve as prerequisites for selection of $R$ to ensure the validity of our graph modeling.
 
\begin{lemma}
Let $p$ be a triangle in a hyper triangulation from the space $\mathcal{HT}(R\leq \sqrt{3}r,\mathcal{L})$. Then, $\mathcal{V}_r(p)\neq \emptyset$. \label{lemma:non-empty-visibility}
\end{lemma}
\begin{proof} See \cite[Lemma 1]{abedi2024indoor}.
\end{proof}Accordingly, to guarantee non-empty visibility areas for all the sub-polygons that partition the layout, the target parameter of hyper-triangulation must be set to $\sqrt{3}r$ or smaller.

\begin{lemma}
Let $p$ be a triangle in a hyper triangulation from $\mathcal{HT}(R\leq r, \mathcal{L})$. Then, $X\in\mathcal{V}_r(p)$ for any point $X$ inside $p$. \label{lemma:Covering-visibility}
\end{lemma}
\begin{proof} See \cite[Lemma 2]{abedi2024indoor}.
\end{proof}

From Lemma \ref{lemma:Covering-visibility}, setting $R
$ in hyper-triangulation to any value smaller than $r$ guarantees that the visibility area of each sub-polygon at least covers all the points within the sub-polygon. 

\begin{lemma}
Let $p$ be a triangle in $\mathcal{HT}(R,\mathcal{L})$, where $\mathcal{L}$ represents a layout without obstacles or a set of $b$ obstacles satisfying $\delta_{\text{min}}=\text{min}(\delta_1, \delta_2,....,\delta_b)\geq \dfrac{\sqrt{3}}{6}r$, with $\delta_i$ denoting the inradius of the convex hull (IRCH) of the $i^\text{th}$ obstacle. Then, for any point $X$ within the visibility area of $p$, it can be concluded that $X$ also lies within the visibility area of any point inside $p$, regardless of the value of $R$.
%Let $p$ be a triangle in $\mathcal{HT}(R,\mathcal{L})$, where $\mathcal{L}$ denotes a layout with no obstacle or a set of $b$ obstacles that satisfy $\delta_{\text{min}}=\text{min}(\delta_1, \delta_2,....,\delta_b)\geq \dfrac{\sqrt{3}}{6}r$, where $\delta_i$denotes the Inradius of  Convex Hull (IRCH) of the $i^{th}$ obstacle. Then, for any point $X$ that lies within the visibility area of $p$,  we conclude that $X$ lies within the visibility area of any point inside $p$, regardless of the value of $R$. %In other words,  $X\in \mathcal{V}_r(p)\Rightarrow  X\in \mathcal{V}_r(Q), ~\forall~ Q\in~p\in \mathcal{HT}(R,\mathcal{L}) $.
\label{Lemma:visibility_to_inside_triangle_no_obstacle}
\end{lemma} \begin{proof} See Appendix A.
\end{proof} Lemma \ref{Lemma:visibility_to_inside_triangle_no_obstacle} states that in a layout without obstacles or with relatively large obstacles, visibility to the vertices of a sub-polygon in a hyper-triangulation guarantees visibility to all points within the sub-polygon, regardless of the value of $R$. In this case, we only need to follow the upper bounds for $R$ established by Lemmas \ref{lemma:non-empty-visibility} and \ref{lemma:Covering-visibility}.
However, in case the environment contains relatively small obstacles,  setting a new upper-bound to $R$ is necessary.

\begin{lemma}
Let $p$ be a triangle in $\mathcal{HT}(R\leq 2x^*, \mathcal{L})$ and $x^*$ be the smallest real positive root of the cubic equation $f(x)=x^3-rx^2+\delta_{\text{min}}^2x+\delta_{\text{min}}^2r$, in which  $\delta_{\text{min}}=\text{min}(\delta_1, \delta_2,....,\delta_b)\leq  \dfrac{\sqrt{3}}{6}r$, where $\delta_i$
denotes the  IRCH of the $i^{th}$ obstacle in $\mathcal{L}$. Then, for any point $X$ within the visibility area of $p$, we conclude that $X$ lies within the visibility area of any point inside $p$. In other words, $X\in \mathcal{V}_r(p), p\in \mathcal{HT}(R\leq 2x^*, \mathcal{L}) \implies X\in \mathcal{V}_r(Q), ~\forall~ Q\in p$. \label{Lemma:visibility_to_inside_triangle_with_obstacle}
\end{lemma} 
\begin{proof} See Appendix B.
\end{proof}
From Lemma \ref{Lemma:visibility_to_inside_triangle_with_obstacle}, the upper-bound to the target parameter $R$ of the hyper-triangulation is determined by $r$ as well as the shape and size of the smallest obstacles in the layout. Further, the obstacle with the smallest IRCH results in the smallest upper bound to $R$. On the other hand, as $r$ grows large, the value of $x^*$ decreases asymptotically to  $\delta_{\text{min}}$. This implies that for an unlimited $r$,  $\delta_{\text{min}}$ sets the upper-bound to $R$. Finally, any hyper-triangulation with a parameter $R$ that satisfies the upper bounds established by Lemmas \ref{lemma:non-empty-visibility}, \ref{lemma:Covering-visibility}, and \ref{Lemma:visibility_to_inside_triangle_with_obstacle} can generate a valid graph that correctly represents the environment.

\subsection{Graph construction}

Assume that a total of $T\geq 1$ realizations are used to represent the given stochastic environment. Let $p_i^{(t)}$ be the $i^{\text{th}}$ triangle of $t^{\text{th}}$ realization, wherein $i=1,2,...,M^{(t)}$ and $t=1, 2,..., T$. Here, $M^{(t)}$ denotes the total number of triangles that partitions the $t^{\text{th}}$ realization of the layout, i.e., $\mathcal{L}^{(t)}$, using a hyper-triangulation $\mathcal{HT}(R,\mathcal{L}^{(t)})$, with $R$ chosen in accordance with the upper bounds outlined in Lemmas
\ref{lemma:non-empty-visibility}, \ref{Lemma:visibility_to_inside_triangle_no_obstacle}, and \ref{Lemma:visibility_to_inside_triangle_with_obstacle}. This choice of $R$ is regarded as valid for all realizations, assuming that the positions of the obstacles may vary while their shapes and sizes remain constant. Using Definitions \ref{point_visibility.def} and \ref{polygon_visibility.def}, the graph model is defined as follows.
\begin{definition}$\hspace{-0.1cm}\bf{Partition\hspace{-0.1cm}-\hspace{-0.1cm}based~Visibility~(PV)~ graph}$ refers to a simple unweighted graph whose nodes represent the sub-polygons (triangles) $p_i^{(t)}$ for $1 \leq t \leq T$ and $1 \leq i \leq M^{(t)}$. Two different nodes  $p_i^{(t)}$ and $p_j^{(t')}$ in this graph are adjacent if and only if $\mathcal{V}_r(p_i^{(t)})\cap \mathcal{V}_r(p_j^{(t')})\neq \emptyset$. %for $1\leq t,t' \leq T$, $1 \leq i \leq M^{(t)}$, and $1 \leq j \leq M^{(t')}$.
\label{definition:PV-graph}
\end{definition}
Algorithm \ref{tabel:PV-graph creation} illustrates the steps of creating a PV graph based on Definition \ref{definition:PV-graph}. The structure of the PV graph is determined by the layout and obstacle shapes, the statistical distribution of obstacle locations, the AP range $r$, and the selected value of $R$.  
  Fig. \ref{fig:graph-large} illustrates the schematic process of creating a PV graph in an example stochastic environment.
Fig. \ref{fig:graph-large}(a) shows the example layout with a square-shaped obstacle located randomly within the area limited by the dash-dotted line. Fig. \ref{fig:graph-large}(b) illustrates the set of $T$ realizations of the environment and their corresponding hyper-triangulation. Finally, Fig. \ref{fig:graph-large}(c) depicts the resulting nodes of the PV graph that represents the sub-polygons (triangles) of the layout's realizations. Therefore, the size of the PV graph can be calculated as $\sum_{t=1}^TM^{(t)}$.

\setcounter{table}{0} 
\renewcommand{\tablename}{Algorithm}
\begin{table}[!t]
\caption{\small{Creating the PV graph (Definition \ref{definition:PV-graph} )}}\label{tabel:PV-graph creation}
\vspace{-0.3cm}
\begin{center}
\hspace{0cm}\begin{tabular}{ l | l}\hline\hline

{\footnotesize{~}}&
{\footnotesize{{\footnotesize{\textbf{Inputs:} }}}} \\

{\footnotesize{~}}&
{\footnotesize{{\footnotesize{~$r$ ~~~~~~~~~~~~~~~~~~~~~~~$\%$ The range of APs }}}} \\

{\footnotesize{~}}&
{\footnotesize{{\footnotesize{~$\nabla_i$; $i=1,2,...b$ ~~~~~$\%$ The 2D shape of the $i^{\text{th}}$ obstacle}}}} \\

{\footnotesize{~}}&
{\footnotesize{{\footnotesize{~$T$~~~~~~~~~~~~~~~~~~~~~~~ $\%$ The  total number of realizations}}}} \\

{\footnotesize{~}}&
{\footnotesize{{\footnotesize{~$\mathcal{L}^{(t)}$; $t=1,2,...T$ ~~$\%$ The $t^\text{th}$ realization of the layout}}}} \\

{\footnotesize{~}}&
{\footnotesize{{\footnotesize{\textbf{Functions:} }}}} \\

{\footnotesize{~}}&
{\footnotesize{{\footnotesize{$\mathcal{V}_r(p)$~~~~~~~~~~~~~~$\%$~Returns the visibility area of sub-polygon $p$}}}} \\

{\footnotesize{~}}&
{\footnotesize{{\footnotesize{$\mathbb{IRCH}(\nabla)$~~~~~~~~~ $\%$~Returns the IRCH of the obstacle $\nabla$}}}} \\

{\footnotesize{~}}&
{\footnotesize{{\footnotesize{$\mathbb{UCal}(\delta_1, \delta_2,..)$~~ $\%$~Upperbound Calculation via Lemmas \ref{lemma:non-empty-visibility}, \ref{lemma:Covering-visibility}, \ref{Lemma:visibility_to_inside_triangle_with_obstacle}}}}}

\\
{\footnotesize{~}}&{\footnotesize{{\footnotesize{$\mathbb{Conn}\big(p,q\big)$  ~~~~~~ $\%$~Connecting the nodes $p$ and $q$ in $\mathcal{G}$ }}}} \\

{\footnotesize{~}}&{\footnotesize{{\footnotesize{~~~~~~~~~~~~~~~~~~~~~~$\%$ if $\mathcal{V}_r(p)\cap \mathcal{V}_r(q)\neq \emptyset$ }}}} \\

{\footnotesize{~}}&
{\footnotesize{{\footnotesize{$\mathcal{HT}(.)$~~~~~~~~~~~~~ $\%$~Hyper triangulation, Definition \ref{hyper-triangulation.def} }}}}

\\

\hline

{\footnotesize{1}}&
{\footnotesize{{\footnotesize{{$\delta_i\longleftarrow \mathbb{IRCH}(\nabla_i)$; $i=1,2,...b$ } }}}} \\

{\footnotesize{2}}&
{\footnotesize{{\footnotesize{{$R'\longleftarrow \mathbb{UCal}(r,\delta_1,...\delta_b)$}  }}}} \\

{\footnotesize{3}}&
{\footnotesize{{\footnotesize{$p_i^{(t)}\longleftarrow \mathcal{HT}(R<R',\mathcal{L}^{(t)})$}; $i=1,2,..., M^{(t)}$; $t=1,2,..., T$   }}} \\

{~}&
{\footnotesize{{\footnotesize{~~~~~~~~~~~~~~~~~~~~~~~~~~~~~~~~$\%$ Creating the nodes of $\mathcal{G}$}}}} \\

{\footnotesize{4}}&{\footnotesize{{\footnotesize{$\mathbb{Conn}\big(p_i^{(t)}, p_j^{(t')} \big)$;  $ t=1,2,..., T$; $t'=1,2,..., T$;  }}}} \\

   {\footnotesize{~}}&
{\footnotesize{{\footnotesize{~~~~~~~~~~~~~~~~~~~~~~~~~ $i=1,2,..., M^{(t)}$; $j=1,2,..., M^{(t')}$  }}}} \\

\hline

   {\footnotesize{~}}&
{\footnotesize{{\footnotesize{\textbf{Output:} }}}} \\

   {\footnotesize{~}}&
{\footnotesize{{\footnotesize{$\mathcal{G}$ $~~~~~~~~~~~~~~~~~\%$ The PV graph }}}} \\

\hline
\hline

\end{tabular}
\end{center}
\vspace{-0.1cm}
\end{table}

Fig. \ref{fig:graph-modeling-example} illustrates an example of creating a PV graph via Algorithm \ref{tabel:PV-graph creation} for a stochastic environment and the maximum range $r=\infty$. Fig. \ref{fig:graph-modeling-example}a) shows the example layout and the square-shaped obstacle located anywhere within the dash-dotted area. Setting the number of realizations $T=2$, the obstacle is shown in two different locations as in Fig. \ref{fig:graph-modeling-example}b) and \ref{fig:graph-modeling-example}c). By using the convex partitioning method, the layout at each realization is partitioned into $M^{(1)}=M^{(2)}=8$ sub-polygons which are reflected as $16$ nodes of the resulting PV graph in Fig. \ref{fig:graph-modeling-example}d)\footnote{To simplify the graph for presentation purposes, we used convex partitioning rather than a hyper-triangulation.}. Then, by creating the visibility area of each sub-polygon, we determine the edges of the PV graph.  For example, since $\mathcal{V}_\infty (p_1^{(2)})\cap \mathcal{V}_\infty (p_4^{(1)})\neq \emptyset$ and $\mathcal{V}_\infty (p_1^{(1)})\cap \mathcal{V}_\infty (p_4^{(1)})= \emptyset$, the nodes $p_1^{(2)}$ and $p_4^{(1)}$ are adjacent while the nodes $p_1^{(1)}$ and $p_4^{(1)}$ are not adjacent in the PV graph.

\section{Graph Processing for Optimal Access Point Deployment }
\label{sec:4}

An analysis of a PV graph provides invaluable information, including possible AP placements that ensure LoS coverage, as well as a useful tool for evaluating whether the placements are optimal. Moreover, PV graphs enable the consideration of diverse network requirements and KPIs, making them an effective interface between our graph analysis methods and AP placement in real-world environments.

\begin{figure}[!t]
\centering
%\hspace{-1cm}
\advance\leftskip-1cm
\advance\rightskip-1cm
\includegraphics[width=8cm]{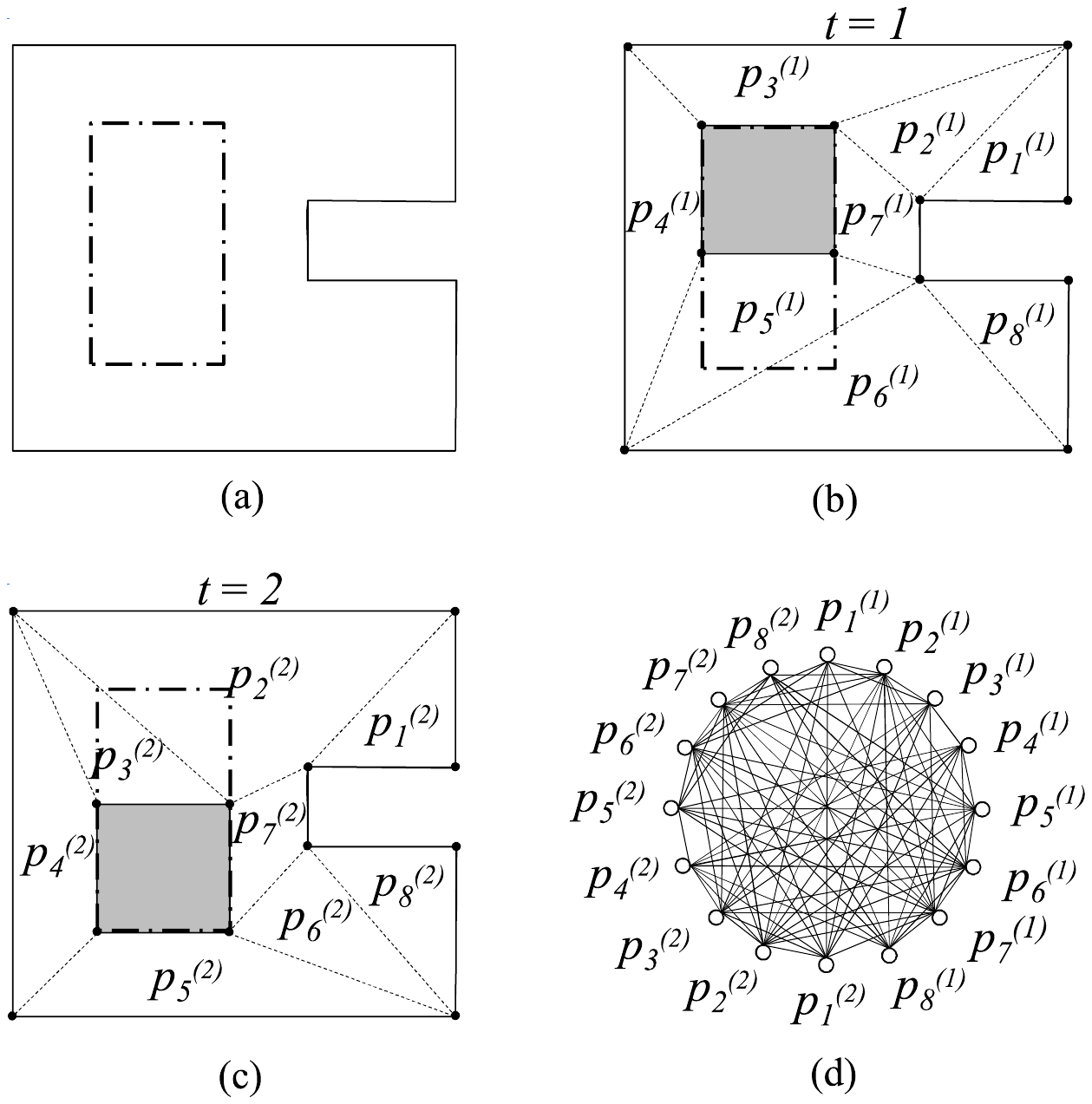} %{vtc_fig0.eps}
    \vspace{-2.8cm}\caption{ Creating the visibility graph for the example layout and a square-shaped obstacle with maximum $T=2$ number of realizations and range $r=\infty$; a) The obstacle  randomly located within the dash-dotted area,
    b) the first and c) the second  layout's realization and the corresponding convex partitioning, and d) the resulting PV graph. }\label{fig:graph-modeling-example}
\vspace{-0.2cm}
\end{figure}
  
\subsection{Graph processing to  ensure a seamless coverage}
Here, we aim to minimize the number of APs required and determine their exact locations to ensure a seamless LoS coverage of the stochastic environment. Assume creating a PV graph that represents a stochastic environment as described in Algorithm \ref{tabel:PV-graph creation}. We bridge between AP placement and the clustering of the nodes of this graph  through the following theorem;

\begin{Theorem}\label{Theorem:clique_partitioning}
Let the partitioning of the nodes of a PV graph into $g$ cliques be $c_1$, $c_2$,..., and $c_g$, such that $\mathcal{V}_r(c_j)\neq \emptyset$ for $j=1,2,... g$.
Deploying a set of $g$ APs, with one placed anywhere within each $\mathcal{V}_r(c_j)$, ensures LoS coverage of the environment across all realizations.
\end{Theorem}

\begin{proof}
 See Appendix C.
\end{proof}

Theorem \ref{Theorem:clique_partitioning} suggests that clustering the nodes of the PV graph into the minimum number of cliques identifies the minimum number of APs and their optimal locations required to ensure seamless LoS coverage of the stochastic environment. Building on this implication, we proposed the Maximal Clique Clustering (MCC) technique as a method to partition the nodes of the PV graph into the smallest possible number of cliques \cite{abedi2021visible}. In the MCC method, we iteratively sort the nodes of the graph based on their degrees in an ascending order and add the nodes to the cluster one by one as long as they meet two condition; I) the node is fully connected to all the other nodes in the cluster to form a new clique and II) the new clique has non-empty visibility area. By surfing through all the nodes of the PV graph, the maximal cluster is derived. Finally, we disconnect the maximal cluster from the PV graph and start over the algorithm until no node in the PV graph is left.

Fig. \ref{fig:maximal-dynamic} illustrates the steps of the MCC method applied to the PV graph shown in Fig. \ref{fig:graph-modeling-example}d).
As shown in Fig. \ref{fig:maximal-dynamic}a), the first cluster is initialized by selecting the minimum-degree node, $p_8^{(2)}$.
 The remaining nodes of the PV graph are then added to this cluster one by one in ascending order of degree, provided that each addition forms a clique with the existing nodes in the cluster and has a non-empty visibility area. As a result, this process forms the maximal clique cluster, highlighted in green. The cluster is then removed from the PV graph, and the process is repeated on the remaining nodes, forming the second and third maximal clique clusters, shown in red and blue in Fig. \ref{fig:maximal-dynamic}b) and \ref{fig:maximal-dynamic}c), respectively. 
 The visibility areas of these $g=3$ clusters are depicted in the stochastic environment using the same colors as in Fig. \ref{fig:maximal-dynamic}d).
  As a result, deploying $g=3$ APs, each positioned anywhere within their respective visibility areas, ensures LoS coverage of the stochastic environment, regardless of the location of the obstacle.  However, it remains necessary to demonstrate the optimality of this deployment.

\begin{figure}[!t]
\centering
%%%%%%%%%%%%%%%%%figure%%%%%%%%%%%%%%%%%%%%%%
%\hspace{-1cm}
\advance\leftskip-0.8cm
\advance\rightskip-1cm
\includegraphics[width=9cm]{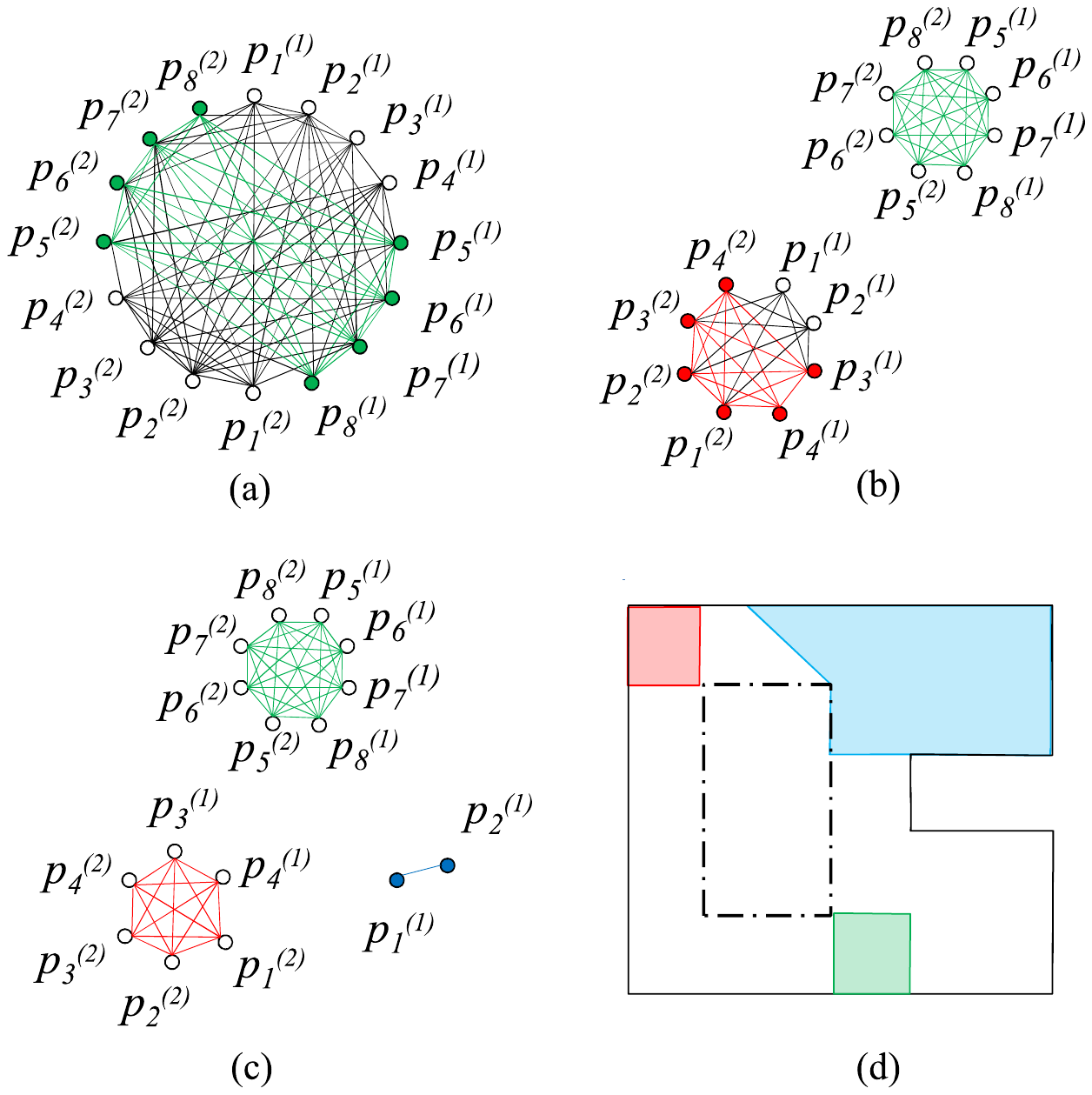} %{vtc_fig0.eps}
    \vspace{-2.9cm}\caption{ Ensuring the full LoS coverage of the example layout and the square-shaped obstacle with random location via MCC method; a) The visibility graph in Fig. \ref{fig:graph-modeling-example} and clustering the first clique (green), b) clustering the second clique (red), c) clustering the third clique (blue), and d) corresponding visibility area of the  $g=3$ clusters.}\label{fig:maximal-dynamic}
\vspace{-0.2cm}
\end{figure}

\begin{Theorem}\label{Theorem:independent set}

Consider a stochastic environment partitioned using hyper-triangulation with a sufficiently small $R$, along with the corresponding PV graph. The maximum size of an independent set of nodes in the PV graph, known as the graph's independence number, serves as a lower bound for the minimum number of APs needed to ensure seamless LoS coverage of the stochastic environment.
\end{Theorem}

\begin{proof}
Let $h$ be the independence number of the PV graph with  $p'_1, p'_2,..., p'_h$ be an independent set of nodes. Also, assume clustering this graph into $g$ cliques $c_1, c_2,..., c_g$, representing $g$ as a number of APs required for LoS coverage of the stochastic environment. Since the nodes $p'_1, p'_2,..., p'_h$ are pairwise disjoint, each one of the clique clusters may contain only one of the independent nodes, and therefore $h\leq g$. 
\end{proof}

According to Theorem \ref{Theorem:independent set}, the independence number of a PV graph serves as a criterion for assessing the optimality of the MCC algorithm in terms of the number of APs.
Analyzing the PV graph in Fig. \ref{fig:graph-modeling-example}a), we identify $h=3$ nodes $p_1^{(1)}, p_8^{(2)}$, and $p_4^{(2)}$ as an independent set.
Thus, according to Theorem \ref{Theorem:independent set}, the $g=h=3$ APs determined in Fig. \ref{fig:maximal-dynamic}d are optimal for ensuring LoS coverage.

\renewcommand{\tablename}{Algorithm}
\begin{table}[!t]
\caption{\small{Maximum Clique  Packing (MCP)}}\label{tabel:Maximum_Clique_partitioning}
\vspace{-0.2cm}
\begin{center}
\hspace{0cm}\begin{tabular}{ l | l}\hline\hline

{\footnotesize{~}}&
{\footnotesize{{\footnotesize{\textbf{Input:} ~~ }}}} \\

{\footnotesize{~}}&
{\footnotesize{{\footnotesize{$\mathcal{G}[p_1^{(1)},p_2^{(1)},...,p_{M^{(T)}}^{(T)}]$ ~~~$\%$ The PV graph; Algorithm \ref{tabel:PV-graph creation} }}}} \\

{\footnotesize{~}}&
{\footnotesize{{\footnotesize{$\alpha_{\text{GAP}}$ ~~~~~~~~~$\%$  Tolerable LoS CGAP, where $0<\alpha_{\text{GAP}}<1$}}}} \\

{\footnotesize{~}}&
{\footnotesize{{\footnotesize{\textbf{Functions:} }}}} \\

{\footnotesize{~}}&
{\footnotesize{{\footnotesize{$\mathcal{V}_r(c)$~~~~~~~~$\%$~Returns the visibility area of the set of nodes $c$}}}} \\

{\footnotesize{~}}&
{\footnotesize{{\footnotesize{$\mathbb{DescSort}(\mathcal{G})$~ $\%$~Sorts the nodes of  $\mathcal{G}$ in descending order}}}} \\

{\footnotesize{~}}&
{\footnotesize{{\footnotesize{$\mathcal{G}.\mathbb{isClique}(c)$~ $\%$~True if the set of nodes $c$ forms  a clique in $\mathcal{G}$}}}}

\\

{\footnotesize{~}}&
{\footnotesize{{\footnotesize{$\mathcal{G}.\mathbb{RealizIdx}(p)$~$\%$~Returns the realization index of $p$ in $\mathcal{G}$}}}}

\\

{\footnotesize{~}}&
{\footnotesize{{\footnotesize{$\mathcal{G}.\mathbb{count}(t)$~~~~~~$\%$~Counts the nodes with  realization index $t$ }}}}

\\

\hline

{\footnotesize{1}}&{\footnotesize{{\footnotesize{$C\longleftarrow \emptyset$  ~$\%$ Initialize the set of clique clusters}}}} \\

{\footnotesize{2}}&
{\footnotesize{{\footnotesize{{\bf{While}} ~~~$\mathcal{G}\neq \emptyset$: }}}} \\

{\footnotesize{3}}&{\footnotesize{{\footnotesize{~~~~$\mathcal{G}\longleftarrow\mathbb{DescSort}(\mathcal{G})~ $ }}}} \\

   {\footnotesize{4}}&
{\footnotesize{{\footnotesize{~~~~$c\longleftarrow \emptyset$ ~~~~~~$\%$ Clique cluster initialization}}}} \\

 {\footnotesize{5}}&{\footnotesize{{\footnotesize{~~~~{\bf{For}}~ $p$ \bf{in} $\mathcal{G}$: }}}} \\
  {\footnotesize{6}}&{\footnotesize{{\footnotesize{~~~~~~~~$c'\longleftarrow c \cup p$  $\%$ adding the node $p$ to the cluster}}}}  \\
 {\footnotesize{7}}&
{\footnotesize{{\footnotesize{~~~~~~~~{\bf{If}} ~~$\mathcal{G}.\mathbb{isClique}(c') $ \bf{and} $\mathcal{V}_r(c')\neq \emptyset$: }}}} \\

{\footnotesize{8}} & {\footnotesize{~~~~~~~~~~~~~$c\longleftarrow  c'$, $t\longleftarrow \mathcal{G}.\mathbb{RealizIdx}(p)$  }}\\

 %{\footnotesize{9}}&
%{\footnotesize{{\footnotesize{~~~~~~~~{\bf{For}} ~~$t=1$~~{\texttt{to}}~~$T$:}}}} \\

 {\footnotesize{9}}&
{\footnotesize{{\footnotesize{~~~~~~~~~~~~{\bf{If}} ~$[\mathcal{G}-c].\mathbb{count}(t)<\alpha_{\text{GAP}} M^{(t)}$: }}}} \\

{\footnotesize{10}}&
{\footnotesize{{\footnotesize{~~~~~~~~~~~~~~~~$\mathcal{G}\leftarrow \mathcal{G}-[p_1^{(t)}, p_2^{(t)},..., p_{M^{(t)}}^{(t)}]$ }}}} \\

{\footnotesize{11}}&
{\footnotesize{{\footnotesize{~~~~$C\leftarrow C\cup c$ }}}} \\

{\footnotesize{12}}&
{\footnotesize{{\footnotesize{~~~~$\mathcal{G}\leftarrow \mathcal{G}-c$ }}}} \\

\hline

  {\footnotesize{~}}&
{\footnotesize{{\footnotesize{\textbf{Output:} }}}} \\

   {\footnotesize{~}}&
{\footnotesize{{\footnotesize{$C$ $~~~~~~~~~~~~~~~~~\%$ The set of clique clusters }}}} \\

\hline

\end{tabular}
\end{center}
\vspace{-0.5cm}
\end{table}

\subsection{Graph processing  to ensure a tolerable coverage gap}

Our objective is to minimize the number of APs required and determine their precise locations to ensure LoS coverage for at least a specified percentage of the stochastic environment, regardless of obstacle positions. In other words, our goal is to ensure that the AP deployment provides LoS coverage for at least $(1-\alpha_{\text{GAP}})\mathcal{A}$ of the stochastic environment, regardless of the exact obstacle locations, where $\mathcal{A}$ represents the total area of the environment, and $\alpha_{\text{GAP}}$ denotes the tolerable LoS Coverage Gap Area Proportion (CGAP).

Algorithm \ref{tabel:Maximum_Clique_partitioning} details the steps in determining the minimum number of APs and their exact locations to ensure a tolerable LoS CGAP in a stochastic environment, referred as Maximum Clique Packing (MCP) method. In this algorithm, at each step, the largest clique in the entire PV graph is identified. To achieve this, given $\alpha_{\text{GAP}}$, we first generate the PV graph of the stochastic environment using Algorithm \ref{tabel:PV-graph creation} and then sort its nodes in descending order of their degree.
Starting with the first node (highest degree), nodes are sequentially added to the cluster as long as the cluster forms a clique with a non-empty visibility area, continuing until the maximum-sized clique cluster is formed. 
%In this process over the PV graph, as soon as only $\alpha_{\text{GAP}} M^{(t)}$ or less number of nodes remain from any realization $t$ for $1\leq t\leq T$, we remove all the remaining nodes associated to $t$. 
In this process over the PV graph, once the number of nodes remaining from any realization $t$, $1\leq t\leq T$,  falls to $\alpha_{\text{GAP}} M^{(t)}$ or less, we remove all the remaining nodes associated to $t$. 
This ensures that the remaining nodes associated with $t$ are excluded from subsequent clique clusters, as their corresponding sub-polygons in the environment do not require LoS coverage. As soon as the maximum clique  is derived, we remove it from the PV graph and start over the algorithm until no node in the PV graph is left.
Thus, placing one AP within the visibility area of each clique cluster guarantees LoS coverage for at least $(1-\alpha_{\text{GAP}})\mathcal{A}$ of the stochastic environment, irrespective of obstacle locations.

  \pagestyle{empty}

%\vspace{0mm}
%\section{Numerical Results}
%\label{sec:4}
%\vspace{-0mm}

Fig. \ref{fig:maximum-dynamic} illustrates the steps of the MCP method, detailed in Algorithm \ref{tabel:Maximum_Clique_partitioning}, applied to the stochastic environment shown in Fig. \ref{fig:graph-modeling-example} with the tolerable CGAP set to $\alpha_{\text{GAP}}=0.25$. 
By sorting the nodes of the PV graph in Fig. \ref{fig:graph-modeling-example}d) in descending order based on degree and adding the nodes sequentially to the cluster, Fig.  \ref{fig:maximum-dynamic}a) shows the first maximum clique with a green color. Removing this cluster from the PV graph and sorting the remaining nodes in descending order, Fig. \ref{fig:maximum-dynamic}b) illustrates the second maximum clique cluster in red color. Removing this cluster from the remaining PV graph, we end up with $2$ remaining nodes at each realization $t=1,2$. Here, since $M^{(1)}= M^{(2)}=8$, the number of remaining nodes at realization $t$ is not bigger than $\alpha_{\text{GAP}} M^{(t)}=2$ for $t=1,2$, and therefore, we leave these four remaining nodes un-clustered in Fig. \ref{fig:maximum-dynamic}c). As a result, the two clique clusters in Fig. \ref{fig:maximum-dynamic}c) represent two APs, each one deployed anywhere inside
the clique visibility areas shown by similar colors in Fig. \ref{fig:maximum-dynamic}d). 
Thus, these $l=2$ APs ensure LoS coverage for at least $(1-\alpha_{\text{GAP}})$ of the stochastic area, regardless of the obstacle's location.

\begin{figure}[!t]
\centering
%%%%%%%%%%%%%%%%%figure%%%%%%%%%%%%%%%%%%%%%%
%\hspace{-1cm}
\advance\leftskip-1.4cm
\advance\rightskip-1cm
\includegraphics[width=9cm]{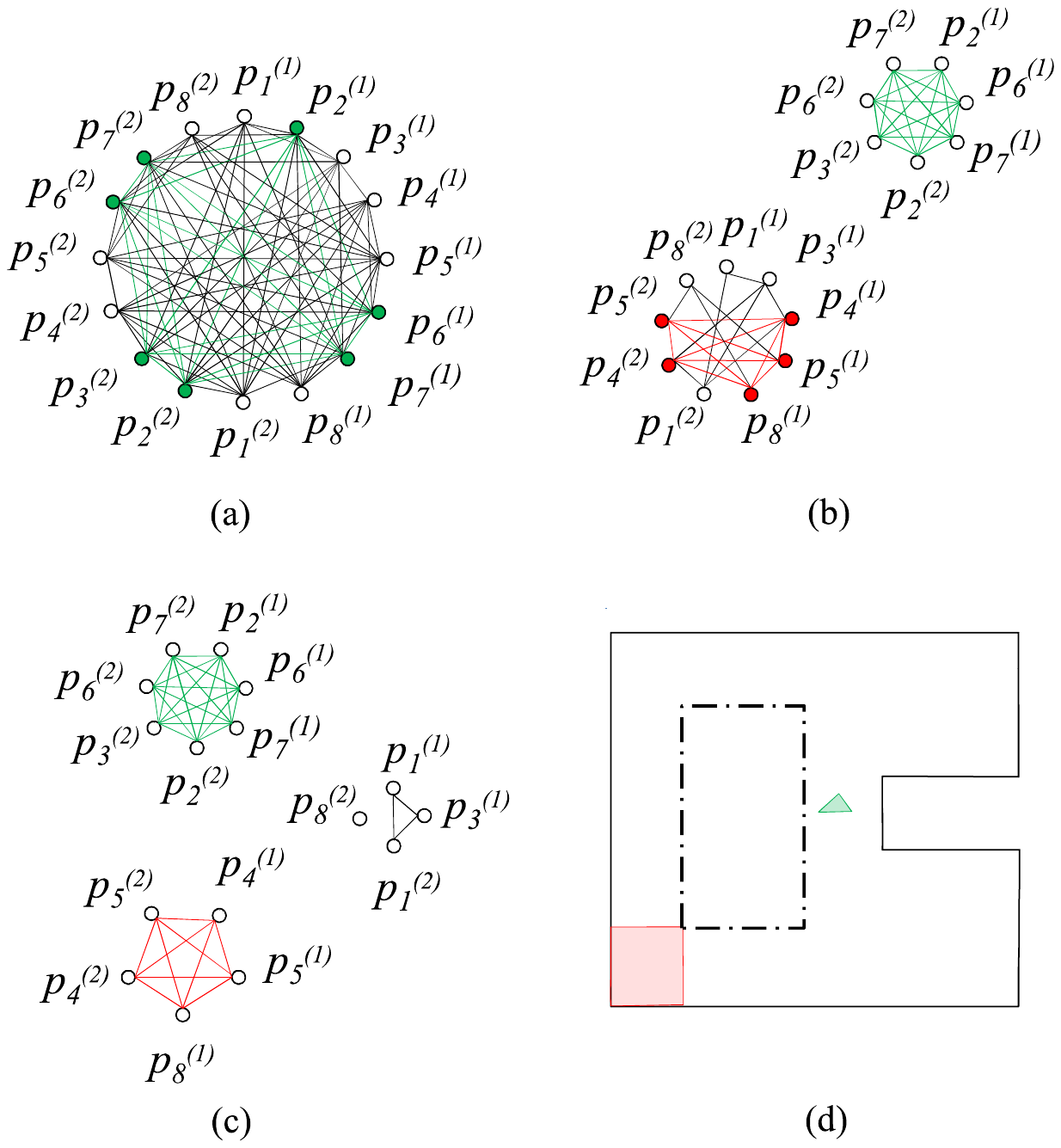} %{vtc_fig0.eps}
    \vspace{-2.4cm}\caption{Performance of MCP method to ensuring a tolerable LoS CGAP $\alpha_{\text{GAP}}=0.25$ for the example layout and the square-shaped
obstacle with random location; a) The PV graph determined in Fig. \ref{fig:graph-modeling-example} and clustering the first clique (green), b) clustering the second clique (red), c) representation of $l=2$ clusters while leaving the remaining 4 nodes un-clustered, and d) corresponding visibility areas of the two clique clusters, representing the potential locations for $l=2$ APs.}\label{fig:maximum-dynamic}
\vspace{-0.2cm}
\end{figure}

\vspace{0mm}
\section{Simulation Results}
\label{sec:5}
\vspace{-0mm}

\begin{figure*}[!t]
\centering
%\hspace{-1cm}
\advance\leftskip-1cm
\advance\rightskip-1cm
\includegraphics[width=20cm]{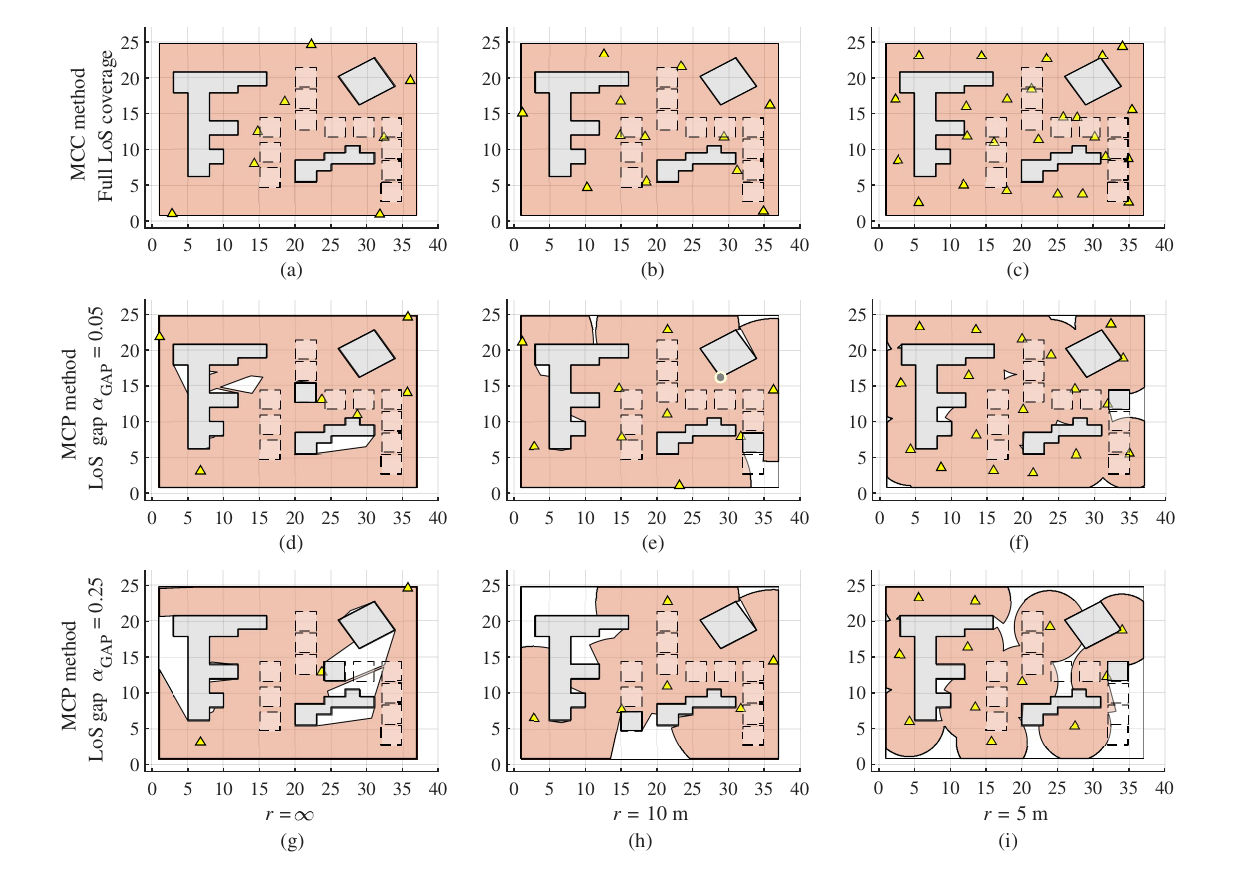} %{vtc_fig0.eps}
    \vspace{-0.7cm}\caption{ Performance analysis of the proposed methods in determining the minimum number of APs (yellow triangles) and their exact locations in a rectangle-shaped environment with $b=4$ obstacles (gray polygons) from which one of them have stochastic/dynamic location (dashed polygons). Ensuring a LoS coverage via MCC method with AP range a) $r=\infty$, b) $r=10$\,m, and c) $r=5$\,m. Ensuring a tolerable LoS coverage area gap $\alpha_{\text{GAP}}=0.05$ via MCP method with AP range d) $r=\infty$, e) $r=10$\,m, and f) $r=5$\,m. Ensuring a tolerable LoS coverage area gap $\alpha_{\text{GAP}}=0.25$ via MCP method with AP range g) $r=\infty$, h) $r=10$\,m, and i) $r=5$\,m.}\label{fig:simulation}
\vspace{-0.2cm}
\end{figure*}

We evaluate the performance of the MCC and MCP methods in determining the minimum number of APs and their precise locations to ensure LoS coverage in a \mbox{24\,m $\times$36\,m} stochastic environment containing $b=4$ obstacles of different shapes and sizes, as shown in Fig. \ref{fig:simulation}. 
In this environment, the three larger obstacles have fixed, deterministic locations, while the smallest square-shaped obstacle, measuring \mbox{2.5\,m $\times$ 2.5\,m}, 
has $T=12$ possible locations shown in the figure.
%has a stochastic location with a spatial distribution spanning from the middle to the lower-right part of the environment.
%In this setting, we considered $T=12$ realizations for the location of the square-shaped obstacle and 
We partitioned the environment using the hyper-triangulation \mbox{ $\mathcal{HT}(R=2\delta_{\text{min}}=2.5\,\text{m}, \mathcal{L})$} that results in the total of $\sum_{t=1}^{T=12}M^{(t)}=13386$ triangles representing the number of nodes in the equivalent PV graph. Here,  $\delta_{\text{min}}$ denotes the  IRCH of the square-shaped obstacle and therefore the selection of the parameter as $R=2\delta_{\text{min}}$ satisfies the criteria described in Lemmas \ref{lemma:non-empty-visibility}, \ref{lemma:Covering-visibility}, and \ref{Lemma:visibility_to_inside_triangle_with_obstacle}.  

By setting the AP range $r=\infty$, Fig. \ref{fig:simulation}a) represents the resulting $g=8$ APs and their locations by applying the MCC method over the equivalent PV graph. Similarly,  running the MCC method over the PV graph with  AP range $r=10$\,m and $r=5$\,m, we determine $g=12$ and $g=25$ APs whose locations are shown in Fig. \ref{fig:simulation}b) and  \ref{fig:simulation}c) respectively. In all these three figures, it is observed that the LoS link from a user anywhere in the environment and at least one AP within it's range is guaranteed regardless of the square-shaped obstacle positioned at any of the $T=12$ dashed-marked locations.

By applying the MCP method with the tolerable maximum coverage gap $\alpha_{\text{GAP}}=0.05$ over the PV graph with AP rang $r=\infty$, $r=10$\,m, and $r=5$\,m, we derive $g=6$, $g=9$, and $g=18$ APs with their exact locations as shown in Figs. \ref{fig:simulation}d),  \ref{fig:simulation}e), and  \ref{fig:simulation}f), respectively.  
All these figures show that the resulting AP placements guarantees LoS link between a user within at least $(1-\alpha_{\text{GAP}})=95\%$ of area and at least one AP within it's range, regardless of the location of the square-shaped obstacle. Similarly, by apply the MCP method with maximum coverage gap $\alpha_{\text{GAP}}=0.25$ over the PV graph with AP rang $r=\infty$, $r=10$\,m, and $r=5$\,m, we derive $g=3$, $g=6$, and $g=13$ APs with their exact locations as shown in Figs. \ref{fig:simulation}g),  \ref{fig:simulation}h), and  \ref{fig:simulation}i), respectively.  
All these figures show that the resulting AP placements guarantees LoS link between a user within at least $(1-\alpha_{\text{GAP}})=75\%$ of the area and at least one AP within it's range, regardless of the location of the square-shaped obstacle. 

An overall report on the results of Fig. \ref{fig:simulation} is summarized in Table \ref{tab:Simulation_figure_report}. It is observed here that applying the MCP method with $\alpha_{\text{GAP}}=0.05$ over the environment, the minimum LoS coverage for around $97\%$ of the area is achieved for all three AP ranges. It is while by setting $\alpha_{\text{GAP}}=0.25$, the minimum  LoS coverage of  $86\%$, $83\%$, and $79\%$ of the area are provided. It is also observed that for a certain AP range, the number of APs required decreases by  $25\%$ to $30\%$ to guarantee LoS coverage for $95\%$ of the area  in comparison to the number of APs required to ensure $100\%$ LoS coverage.

\setcounter{table}{0} 
\renewcommand{\tablename}{Table}
\begin{table}[!t]
\centering
\caption{Performance metric of MCC and MCP methods over the stochastic environment shown in Fig. \ref{fig:simulation}}
\label{tab:Simulation_figure_report}
\begin{tabular}{|c|c|c|c|}
\hline
 
Method          & range  $r$ [m]        & Nubmer of APs $g$          & Coverage            \\ \hline \hline
\multirow{3}{*}{MCC } & $\infty$   & 8    & 100$\%$    \\ \cline{2-4} 
                      & 10    & 12    & 100$\%$   \\ \cline{2-4} 
                      & 5   & 25   & 100$\%$  \\ \hline \hline
\multirow{3}{*}{\shortstack{MCP \\ $\alpha_{\text{GAP}}=0.05$}} & $\infty$    & 6    & 97$\%$   \\ \cline{2-4} 
                      & 10    & 9   & 98$\%$    \\ \cline{2-4} 
                      & 5    & 18    & 97$\%$    \\ \hline \hline
\multirow{3}{*}{\shortstack{MCP \\ $\alpha_{\text{GAP}}=0.25$}} & $\infty$    & 3   &  86$\%$    \\ \cline{2-4} 
                      & 10   & 6    & 83$\%$    \\ \cline{2-4} 
                      & 5   & 12    & 79$\%$    \\ \hline
\end{tabular}

\end{table}

\color{black}{}

      % For better layout:
%\newpage
\pagestyle{empty}

\vspace{0mm}
\section{Conclusion}
\label{sec:6}

This paper proposed a novel method for ensuring LoS connectivity in stochastic environments with randomly positioned obstacles by optimizing wireless AP placement. To this end, the proposed method applied a graph-based framework, where the environment was partitioned using hyper-triangulation, and AP placement was optimized via maximal clique clustering and maximum clique packing techniques. To control the computational complexity of the algorithm, an upper bounds on the hyper-triangulation parameter was derived, which directly influenced the number of nodes in the graph. Unlike conventional stochastic geometry-based approaches, the proposed framework is deterministic and guarantees either a full LoS coverage or a controlled LoS gap while dynamically adapting to variations in obstacle positions. Simulation results demonstrated that the proposed approach can reduce the number of required APs by $25\%$, while maintaining a tolerable $5\%$ coverage gap compared to full LoS deployment. These findings underscore the efficiency and adaptability of the proposed solution in enabling robust high-frequency wireless communication and precise mobile network-based positioning in complex indoor/outdoor environments.

\pagestyle{empty}

\vspace{0mm}
\bibliographystyle{IEEEtran}

\bibliography{main}

\end{document}